\input ./style/arxiv-vmsta.cfg
\documentclass[numbers,compress,v1.0.1]{vmsta}
\usepackage{vtexbibtags}
\usepackage[mathscr]{euscript}
\usepackage{mathrsfs}

\volume{6}
\issue{4}
\pubyear{2019}
\firstpage{479}
\lastpage{494}
\aid{VMSTA144}
\doi{10.15559/19-VMSTA144}
\articletype{research-article}

\DeclareMathOperator*{\argmin}{arg\,min}

\startlocaldefs
\newcommand{\pvtex}{^}
\newcommand{\Xvtex}[1]{X\pvtex{#1}}
\newcommand{\rmd}{\mathrm{d}}
\newcommand{\rme}{\mathrm{e}}
\newcommand{\Lrm}{\mathrm{L}}
\newcommand{\leb}{\lambda}
\newcommand{\Ascr}{\mathscr{A}}
\newcommand{\Bscr}{\mathscr{B}}
\newcommand{\Escr}{\mathscr{E}}
\newcommand{\Fscr}{\mathscr{F}}
\newcommand{\Gscr}{\mathscr{F}^L}
\newcommand{\Hscr}{\mathscr{H}}
\newcommand{\Kscr}{\mathscr{K}}
\newcommand{\Pscr}{\mathscr{P}}
\newcommand{\Sscr}{\mathscr{S}}
\newcommand{\Tscr}{\mathscr{T}}
\newcommand{\Xscr}{\mathscr{X}}
\newcommand{\sig}{\sigma}
\newcommand{\cadlag}{c\`adl\`ag}
\newcommand{\Ebb}{\mathbb{E}}
\newcommand{\Fbb}{\mathbb{F}}
\newcommand{\Nbb}{\mathbb{N}}
\newcommand{\Pbb}{\mathbb{P}}
\newcommand{\Qbb}{\mathbb{Q}}
\newcommand{\Rbb}{\mathbb{R}}
\newcommand{\Gbb}{\mathbb{F}^L}
\newcommand{\aPP}[2]{\ensuremath{\langle #1,#2 \rangle}}

\newtheorem{theorem}{Theorem}[section]
\newtheorem{lemma}[theorem]{Lemma}
\newtheorem{proposition}[theorem]{Proposition}

\theoremstyle{definition}
\newtheorem{definition}[theorem]{Definition}
\newtheorem{assumption}[theorem]{Assumption}
\newtheorem{remark}[theorem]{Remark}

\numberwithin{equation}{section}

\allowdisplaybreaks
\ifdefined\HCode 
\def\index#1{}
\else 
\fi
\endlocaldefs

\begin{document}

\begin{frontmatter}
\pretitle{Research Article}

\title{BSDEs and log-utility maximization for~L\'{e}vy~processes}

\author[a]{\inits{P.}\fnms{Paolo}~\snm{Di Tella}\ead[label=e1]{Paolo.Di\_Tella@tu-dresden.de}}
\author[b]{\inits{H.-J.}\fnms{Hans-J\"{u}rgen}~\snm{Engelbert}\thanksref{cor1}\ead[label=e2]{hans-juergen.engelbert@uni-jena.de}}
\thankstext[type=corresp,id=cor1]{Corresponding author.}
\address[a]{Institute for Mathematical Stochastics, \institution{TU Dresden},
Zellescher Weg 12--14, 01069~Dresden, \cny{Germany}}
\address[b]{Department of Mathematics, \institution{University of Jena},
Ernst-Abbe-Platz 2, 07743~Jena,~\cny{Germany}}


\markboth{P. Di Tella, H.-J. Engelbert}{BSDEs and log-utility maximization for L\'{e}vy processes}

\begin{abstract}
In this paper we establish the existence and the uniqueness of the
solution of a special class of BSDEs for L\'{e}vy processes in the case
of a Lipschitz generator of sublinear growth. We then study a related
problem of logarithmic utility maximization of the terminal wealth in
the filtration generated by an arbitrary L\'{e}vy process.
\end{abstract}

\received{\sday{15} \smonth{2} \syear{2019}}
\revised{\sday{2} \smonth{10} \syear{2019}}
\accepted{\sday{3} \smonth{10} \syear{2019}}
\publishedonline{\sday{28} \smonth{10} \syear{2019}}

\end{frontmatter}

\section{Introduction}%
\label{sec:intro}

This paper consists of two independent but related parts. In the first
part (Section \ref{sec:ex.un.bsdes}), we consider a class of backward stochastic differential equations (from now on BSDEs)\index{BSDEs} (see
\eqref{eq:bsde} below) in the filtration generated by a L\'{e}vy process\index{L\'{e}vy process}
and, for the case of a Lipschitz generator of sublinear growth, we
establish the existence and uniqueness of the solution in Theorem
\ref{thm:ex.un} below. We stress that the proof of Theorem
\ref{thm:ex.un} relies on the predictable representation\index{predictable representation property} property
(from now on PRP) obtained by Di Tella and Engelbert in
\cite{DTE15} (see also \cite{DTE16}) and recalled in Theorem
\ref{thm:prp.lev} below.\goodbreak

A similar class of BSDEs\index{BSDEs} has been considered by Nualart and
Schoutens in \cite{NS01}. Their approach is based on the PRP of
the orthogonalized \emph{Teugels martingales},\index{orthogonalized Teugels martingales} that was obtained in
Nualart and Schoutens \cite{NS00}. However, the PRP of the family
of Teugels martingales\index{Teugels martingales} requires an exponential moment of the L\'{e}vy
measure outside the origin. On the other hand, for the PRP in
\cite{DTE15} no additional assumptions on the L\'{e}vy process\index{L\'{e}vy process} are
needed. Hence, we are able to consider this class of BSDEs\index{BSDEs} for general
L\'{e}vy processes.\index{L\'{e}vy process} Therefore, Theorem \ref{thm:ex.un} below
generalizes \cite[Theorem 1]{NS01} to arbitrary L\'{e}vy processes\index{L\'{e}vy process}
but with a great variety of systems of martingales in place of Teugels
martingales.\index{Teugels martingales}

The second part of the present paper (Section \ref{sec:log.ut}) is
devoted to the study of a problem of logarithmic utility maximization
of terminal wealth. We solve the problem having in mind the dynamical
approach first introduced by Rouge and El Karoui in \cite{RE00}
and further developed by Hu, Imkeller and M\"{u}ller in
\cite{HIM05} in a Brownian setting, which is based on a martingale
optimality principle\index{martingale optimality principle} constructed via BSDEs.\index{BSDEs} In case of a more general
filtration supporting martingales with jumps, the dynamical approach
based on BSDEs\index{BSDEs} has been followed by Morlais in \cite{M09,M10} and
by Becherer in \cite{Be06} to study the problem of exponential
utility maximization of the terminal wealth.\looseness=1

The logarithmic utility maximization problem considered in this paper
is analogous to the one studied in \cite[Section 4]{HIM05} for a
Brownian filtration: We extend the results of
\cite[Section 4]{HIM05} to L\'{e}vy processes\index{L\'{e}vy process} with both Gaussian part
and jumps.

Because of the special structure of the logarithmic utility, it turns
out that, in the present paper, the martingale optimality principle\index{martingale optimality principle} can
be constructed in a direct and independent way, without using BSDEs.\index{BSDEs}
However, as we shall explain (see Remark \ref{rem:BSDEs} below), this
can be also alternatively done using a BSDE of the form of
\eqref{eq:bsde} below.

In \cite{GK03} (see also references therein) Goll and Kallsen have
studied the problem of the logarithmic utility maximization\index{logarithmic utility maximization} in a very
general context (that, in particular, recovers L\'{e}vy processes\index{L\'{e}vy process}),
combining \emph{duality methods} and semimartingale characteristics
calculus. In \cite{GK03}, the authors assume the convexity of the
set in which the admissible strategies\index{admissible strategies} for the optimization problem take
values. However, since for the dynamical approach this further
assumptions is not needed, the set of the constraints considered in the
present paper will be closed and non-necessarily convex. {Other
references for the logarithmic utility maximization\index{logarithmic utility maximization} by the duality
approach are Goll and Kallsen \cite{GK00} and Kallsen
\cite{K00}. We also recall the work by Civitani\'{c} and Karatzas
\cite{CK92} about utility optimization by the duality approach. }

\textit{General setting of the paper.} Let
$(\Omega ,\Fscr ,\Pbb )$ be a complete probability space and
$\Fbb $ be a filtration satisfying the usual conditions. We fix a finite
time horizon $T>0$. We only consider $\Rbb $-valued stochastic processes
on the time interval $[0,T]$. When we say that a process $X$ is a
martingale we implicitly assume that $X$ is \cadlag . We denote by
$\Pscr $ the $\sigma $-algebra of predictable subsets of $[0,T]
\times \Omega $.

By $\Hscr ^{2}$ we denote the space of square integrable martingales
$X$ on $[0,T]$ such that $X_{0}=0$. The space $(\Hscr ^{2},\|\cdot \|
_{2})$ endowed with the norm specified by $\|X\|_{2}^{2}:=\Ebb [X^{2}
_{T}]$, for $X\in \Hscr ^{2}$, is a Hilbert space. We observe that, by
Doob's inequality, the $\|\cdot \|_{2}$-norm is equivalent to the norm
$\|X\|_{\Hscr ^{2}}^{2}:=\Ebb [\sup _{t\in [0,T]}X^{2}_{t}]$,
$X\in \Hscr ^{2}$. However, $(\Hscr ^{2},\|\cdot \|_{\Hscr ^{2}})$ is
not a Hilbert space.

For $X\in \Hscr ^{2}$, we denote by $\aPP{X}{X}$ the predictable square
variation\index{predictable square variation} process of $X$. We recall that $\aPP{X}{X}$ is the unique
predictable increasing process\index{predictable increasing process} starting at zero such that the process
$X^{2}-\aPP{X}{X}$ is an $\Fbb $-martingale. If $X,Y\in \Hscr ^{2}$,
we define $\aPP{X}{Y}$ by polarization and say that $X$ and $Y$ are
orthogonal if $\aPP{X}{Y}=0$.

For $X\in \Hscr ^{2}$, we introduce $L^{2}(X)$ as the space of
$\Fbb $-predictable processes $H$ such that $\int _{0}^{T}H^{2}_{s}\rmd \aPP{X}{X}
_{s}$ is integrable. For $H\in L^{2}(X)$, we denote by $\int _{0}^{
\cdot }H_{s}\rmd X_{s}$ the \emph{stochastic integral\index{stochastic integral}} of $H$ with
respect to $X$. It holds $\int _{0}^{\cdot }H_{s}\rmd X_{s}\in \Hscr
^{2}$, for every $H$ in $L^{2}(X)$. Furthermore, $\int _{0}^{\cdot }H
_{s}\rmd X_{s}$ is characterized in the following way: A martingale
$Z\in \Hscr ^{2}$ is indistinguishable from $\int _{0}^{\cdot }H_{s}\rmd X
_{s}$ if and only if $\aPP{Z}{Y}=\int _{0}^{\cdot }H_{s}\rmd \aPP{X}{Y}
_{s}$, for every $Y\in \Hscr ^{2}$.

For two semimartingales $X$ and $Y$, we denote by $[X,Y]$ the
covariation of $X$ and $Y$, that is, we define $[X,Y]_{t}:=\aPP{X\pvtex c}{Y\pvtex c}+
\sum _{0\leq s\leq t}\Delta X_{s}\Delta Y_{s}$, $t\in [0,T]$, where
$X^{c}$ and $Y^{c}$ denote the continuous local martingale part of
$X$ and $Y$, respectively.

\section{Martingale representation for L\'{e}vy processes\index{L\'{e}vy process}}%
\label{sec:setting}

A L\'{e}vy process\index{L\'{e}vy process} with respect to $\Fbb $ is an $\Rbb $-valued and
$\Fbb $-adapted stochastically continuous \cadlag \ process $L$ such
that $L_{0}=0$, $(L_{t+s}-L_{t})$ is distributed as $L_{s}$ and it is
independent of $\Fscr _{t}$, for all $s,t\geq 0$. In this case,
we say that $(L,\Fbb )$ is a L\'{e}vy process.\index{L\'{e}vy process} By $\Fbb ^{L}$ we
denote the completion in $\Fscr $ of the filtration generated by
$L$. It is well known that $\Fbb ^{L}$ satisfies the usual conditions
(see \cite{W81}). Clearly, $(L,\Fbb ^{L})$ is again a L\'{e}vy
process.\index{L\'{e}vy process}

Let $(L,\Fbb )$ be a L\'{e}vy process.\index{L\'{e}vy process} We denote by $\mu $ the jump
measure of $L$. Recall that $\mu $ is a homogeneous Poisson random
measure on $[0,T]\times \Rbb $ with respect to $\Fbb $ (see
\cite{JS00}, Definition II.1.20). The $\Fbb $-predictable compensator
of $\mu $ is deterministic and it is given by $\leb \otimes \nu $,
where $\leb $ is the Lebesgue measure on $[0,T]$ and $\nu $ is the
L\'{e}vy measure of $L$, that is, $\nu $ is, in particular, a
$\sigma $-finite measure such that $\nu (\{0\})=0$ and $x\mapsto x
^{2}\wedge 1$ is $\nu $-integrable. We call $\overline{\mu }:=
\mu -\leb \otimes \nu $ the compensated Poisson random measure
associated with $\mu $. By ${B}^\sig  $ we denote the Gaussian part
of $L$, which is an $\Fbb $-Brownian motion such that $\Ebb [({B}
^\sig  _{t})^{2}]=\sigma ^{2} t$, where $\sigma ^{2}\geq 0$. By
$\eta \in \Rbb $ we denote the drift parameter of $L$ and we call
$(\eta ,\sigma ^{2},\nu )$ the characteristic triplet\index{characteristic triplet} of $L$.

A function $G$ on $\widetilde{\Omega }:=[0,T]\times \Omega \times \Rbb $
is said to be predictable if $G$ is
$\Pscr \otimes \Bscr (\Rbb )$-measurable. Let $\mathscr{G}^{2}(\mu )$
denote the linear space of the predictable functions $G$\index{predictable ! function} on
$\widetilde{\Omega }$ such that $ \Ebb [\sum _{s\leq T}G^{2}({s,\omega
},\Delta L_{s}(\omega ))1_{\{\Delta L_{s}(\omega )\neq 0\}}]<+\infty
$.

For $G\in \mathscr{G}^{2}(\mu )$, we denote by $
\int _{[0,\cdot ]\times \Rbb }G(s,y)\overline{\mu }(\rmd s,\rmd y)$ the
stochastic integral\index{stochastic integral} of $G$ with respect to $\overline{\mu }$. Recall
that $\int _{[0,\cdot ]\times \Rbb }G(s,y)\overline{\mu }(\rmd s,\rmd y)$
is defined as the unique purely discontinuous $Z\in \Hscr ^{2}$ with
$ \Delta Z_{t}(\omega )=G({t,\omega },\Delta L_{t}(\omega ))1_{\{
\Delta L_{t}(\omega )\neq 0\}} $ (see \cite{JS00}, II.1.27).

Next, for any $f\in L^{\,2}(\nu )$, we introduce the martingale
$X^{f}\in \Hscr ^{2}$ by
\begin{equation*}
X^{f}_{\,t}:=\int _{[0,t]\times \Rbb }f(y)\overline{\mu }(\rmd s,\rmd y),
\quad t\in [0,T].
\end{equation*}

Let $(L,\Fbb )$ be a L\'{e}vy process\index{L\'{e}vy process} with characteristic triplet\index{characteristic triplet}
$(\eta ,\sigma ^{2},\nu )$ on $[0,T]$. Then, for any $f\in L^{2}(
\nu )$, the martingale $X^{f}$ has the following properties:
\bigskip

\textnormal{(i)} For every $f\in L^{2}(\nu )$, $(X^{f},\Fbb )$ is a
L\'{e}vy process\index{L\'{e}vy process} and
\begin{equation*}
\Ebb [(X^{f}_{\,t})^{\,2}]=t\,\int _\Rbb f^{2} (y)\nu (\rmd y)<+
\infty .
\end{equation*}

\textnormal{(ii)} For every $f,g\in L^{2}(\nu )$, $\langle X^{f},X
^{g} \rangle _{\,t}=t\int _\Rbb f(y)g(y)\,\nu (\rmd y)$.

\textnormal{(iii)} For every $f,g\in L^{2}(\nu )$, $X^{f}$ and
$X^{g}$ are orthogonal martingales if and only if $f,g\in L^{\,2}(
\nu )$ are orthogonal functions in $L^{\,2}(\nu )$.
\bigskip

Let $\Tscr $ be an arbitrary subset of $L^{2}(\nu )$. Then we set
%
\begin{equation}
\label{eq:mart.fam}
\Xscr _\Tscr :=\{{B}^\sig  \}\cup \{X^{f},\ f\in \Tscr \}.
\end{equation}

As usual, $\ell ^{2}$ denotes the Hilbert space of sequences
$a=(a^{n})_{n\geq 1}$ of real numbers for which the norm $ \|a\|
_{\ell ^{2}}^{2}:=\sum _{n=1}^{\infty }(a^{n})^{2} $ is finite.
%
\begin{definition}
\label{def:M2}
We denote by $(M^{2}(\ell ^{2}),\|\cdot \|_{M^{2}(\ell ^{2})})$ the
Hilbert space of $\ell ^{2}$-valued $\Fbb $-predictable processes $V$
such that $\|V\|_{M^{2}(\ell ^{2})}^{2}:=\Ebb [\int _{0}^{T}\|V_{s}\|
_{\ell ^{2}}^{2}\rmd s]<+\infty $.
\end{definition}
Let $\Tscr =\{f_{n},\ n\geq 1\}\subseteq L^{2}(\nu )$ be an orthonormal
basis. We remark that the orthogonal sum $\sum _{n=1}^{\infty }\int
_{0}^{\cdot }V^{n}_{s}\rmd X^{f_{n}}_{s}$ converges in (and, therefore,
belongs to) the Hilbert space $(\Hscr ^{2},\|\cdot \|_{2})$ if and only
if $V=(V^{n})_{n\geq 1}$ belongs to $M^{2}(\ell ^{2})$. In this case we
have the isometry $\|\sum _{n=1}^{\infty }\int _{0}^{\cdot }V^{n}_{s}\rmd X
^{f_{n}}_{s}\|_{2}=\|V\|_{M^{2}(\ell ^{2})}$.

For the next theorem, we refer to \cite[Section 4.2]{DTE15}. It
states that the family $\Xscr _\Tscr $, where $\Tscr \subseteq L
^{2}(\nu )$ is an orthonormal basis, possesses the PRP.
%
\begin{theorem}
\label{thm:prp.lev}
Let $(L,\Fbb ^{L})$ be a L\'{e}vy process\index{L\'{e}vy process} with characteristic triplet\index{characteristic triplet}
$(\eta ,\sigma ^{2},\nu )$ on the probability space $(\Omega ,\Fscr
_{T}^{L},\Pbb )$.\index{probability space} Let $\Tscr :=\{f_{n},\ n\geq 1\}$ be an orthogonal
basis of $L^{2}(\nu )$ and let $\Xscr _\Tscr \subseteq \Hscr ^{2}$ be
the associated family of $\Fbb ^{L}$-martingales defined in
\eqref{eq:mart.fam}. Then $\Xscr _\Tscr $ consists of pairwise
orthogonal martingales and every $\Fbb ^{L}$-square integrable
martingale $N$ has the representation
\begin{equation*}
N=N_{0}+\int _{0}^{\cdot }Z_{s}\rmd {B}^\sig  _{s}+\sum _{n=1}^{
\infty }\int _{0}^{\cdot }V^{n}_{s}\rmd X^{f_{n}}_{s},\quad Z\in L^{2}({B}^\sig  ),\ V^{n}\in L^{2}(X^{f_{n}}),\ n\geq 1,
\end{equation*}
where the spaces $L^{2}({B}^\sig  )$ and $L^{2}(X^{f_{n}})$ are
considered with respect to $\Fbb ^{L}$.
\end{theorem}

\section{BSDEs for L\'{e}vy processes\index{BSDEs for L\'{e}vy processes}}%
\label{sec:ex.un.bsdes}

Let $(L,\Fbb ^{L})$ be a L\'{e}vy process\index{L\'{e}vy process} with characteristic triplet\index{characteristic triplet}
$(\eta ,\sigma ^{2},\nu )$. We consider the probability space
$(\Omega ,\Fscr ^{L}_{T},\Pbb )$.\index{probability space} Because of Theorem
\ref{thm:prp.lev}, for each orthogonal basis $\Tscr =\{f_{n},\ n
\geq 1\}$ of $L^{2}(\nu )$, it is natural to consider the following
BSDE:
%
\begin{equation}
\label{eq:bsde}
Y_{t}=\xi +\int _{t}^{T}f(s,Y_{s},{Z_{s},V_{s}})\rmd s-\int _{t}^{T}Z
_{s}\rmd {B}^\sig  _{s}-\sum _{n=1}^{\infty }\int _{t}^{T} V^{n}_{s}\rmd X
^{f_{n}}_{s},
\end{equation}
where $ f:{[0,T]\times \Omega }\times \Rbb \times \Rbb \times \ell
^{2}\longrightarrow \Rbb
$ is a given random function, called \emph{the generator} of BSDE
\eqref{eq:bsde}, and $\xi $ is an $\Fscr ^{L}_{T}$-measurable random
variable. We call the pair $(f,\xi )$ the \emph{data} of BSDE
\eqref{eq:bsde}.

In \cite{NS00} BSDE \eqref{eq:bsde} is studied for Teugels
martingales\index{Teugels martingales} under the further assumption of the existence of an
exponential moment for $L$, that is, $\Ebb [\rme ^{\varepsilon |L_{1}|}]<+
\infty $, for an $\varepsilon >0$. We are going to study BSDE
\eqref{eq:bsde} for arbitrary L\'{e}vy processes\index{L\'{e}vy process} $L$ and for a great
variety of families $\Xscr _\Tscr =\{{B}^\sig  \}\cup \{X^{f_{n}},
n\ge 1\}$, 
not restricting the analysis to Teugels martingales.\index{Teugels martingales}

We introduce the space $\Sscr ^{2}$ by
\begin{equation*}
\textstyle
\Sscr ^{2}:=\big \{\Fbb ^{L}\textnormal{-adapted processes}\ Y\ \textnormal{such that}\ \Ebb [\sup _{t\in [0,T]}Y_{t}^{2}]<+\infty
\big \}.
\end{equation*}

\begin{definition}
\label{def:solution.lip}
A triplet $(Y,Z,V)\in \Sscr ^{2}\times L^{2}({B}^\sig  )\times M
^{2}(\ell ^{2})$ satisfying \eqref{eq:bsde} is a solution of BSDE
\eqref{eq:bsde}.
\end{definition}
We now state the assumptions on the generator $f$ of \eqref{eq:bsde}
which we shall need in the proof of Theorem \ref{thm:ex.un} below.
%
\begin{assumption}
\label{ass:gen}
Let the generator $f$ of BSDE \eqref{eq:bsde} fulfil the following
properties:

(i) For $(y,z,a)\in \Rbb \times \Rbb \times \ell ^{2}$, $f(\cdot ,
\cdot ,y,z,a)$ is an $\Fbb ^{L}$-progressively measurable process.

(ii) There exists a constant $K>0$ and a nonnegative
$\Fbb ^{L}$-progressively measurable process $\gamma $ such that
$\Ebb [\int _{0}^{T}\gamma _{s}^{2}\rmd s]<+\infty $ and
\begin{equation*}
|f(t,y,z,a)|\leq \gamma _{t}+K\Big (|y|+|\sigma ||z|+\|a\|_{\ell ^{2}}
\Big ).
\end{equation*}

(iii) There exists a constant $C>0$ such that
\begin{equation*}
\begin{split}
|f(t,y_{1},z_{1},a_{1})-
&f(t,y_{2},z_{2},a_{2})|
\\
&\leq C\Big (|y_{1}-y_{2}|+|\sigma |\,|z_{1}-z_{2}|+\|a_{1}-a_{2}\|_{\ell
^{2}}\Big ).
\end{split}
\end{equation*}
\end{assumption}
A generator $f$ fulfilling the properties (i), (ii) and (iii) in
Assumption \ref{ass:gen}, will be called \emph{admissible}.
%
\begin{theorem}
\label{thm:ex.un}
Let $\xi \in L^{2}(\Omega ,\Fscr ^{L}_{T},\Pbb )$ and let $f$ be an
admissible generator. Then \textnormal{BSDE} \eqref{eq:bsde} with
data $(f,\xi )$ admits a unique solution $(Y,Z,V)$.
\end{theorem}
\begin{proof}
We denote by $L^{2}(\leb \otimes \Pbb )$ the space of
$\leb \otimes \Pbb $-square integrable adapted processes on
$[0,T]$. We introduce the space $\Kscr ^{2}:=L^{2}(\leb \otimes \Pbb )
\times L^{2}({B}^\sig  )\times M^{2}(\ell ^{2})$ endowed with the norm
$\|\cdot \|_{\Kscr ^{2}}$ defined by $\|\cdot \|_{\Kscr ^{2}}^{2}=\|
\cdot \|^{2}_{L^{2}(\leb \otimes \Pbb )}+\|\cdot \|^{2}_{L^{2}({B}
^\sig  )}+\|\cdot \|^{2}_{M^{2}(\ell ^{2})}$. It is clear that
$(\Kscr ^{2},\|\cdot \|_{\Kscr ^{2}})$ is a Banach space. We now define
the mapping
\begin{equation*}
\Phi :\Kscr ^{2}\longrightarrow \Kscr ^{2}
\end{equation*}
by setting $(Y,Z,V)=\Phi (R,S,P)$, for $(R,S,P)\in \Kscr ^{2}$, where
%
\begin{align}
\label{eq:def.Y}
Y_{t}
&:=\Ebb \left [\xi +\int _{t}^{T}f(s,R_{s},S_{s},P_{s})\rmd s
\Big |\Fscr ^{L}_{t}\right ]\nonumber
\\
&=\Ebb \left [\xi +\int _{0}^{T}f(s,R_{s},S_{s},P_{s})\rmd s\Big |\Fscr
^{L}_{t}\right ]-\int _{0}^{t}f(s,R_{s},S_{s},P_{s})\rmd s,
\end{align}
and $(Z,V)$ is the unique pair in $L^{2}({B}^\sig  )\times M^{2}(
\ell ^{2})$ such that
%
\begin{align}
\label{eq:def.Z.U}
\xi +\int _{0}^{T}
&f(s,R_{s},S_{s}, P_{s})\rmd s\nonumber
\\
={}&
\Ebb \left [\xi +\int _{0}^{T}f(s,R_{s},S_{s},P_{s})\rmd s\right ]+
\int _{0}^{T}Z_{s}\rmd {B}^\sig  _{s}+\sum _{n=1}^{\infty }\int _{0}
^{T}V^{n}_{s}\rmd X^{f_{n}}_{s}\,.
\end{align}
We observe that from Assumption \ref{ass:gen} (ii) and $\xi \in L
^{2}(\Omega ,\Fscr ^{L}_{T},\Pbb )$, the random variable $\xi +\int
_{0}^{T}f(s,R_{s},S_{s},P_{s})\rmd s$ is square integrable, since
$(R,S,P)\in \Kscr ^{2}$. Hence $Y$ is well defined by
\eqref{eq:def.Y}. Moreover, since $\xi +\int _{0}^{T}f(s,R_{s},S_{s},P
_{s})\rmd s\in L^{2}(\Omega ,\Fscr ^{L}_{T},\Pbb )$, the existence of
the unique pair $(Z,V)\in L^{2}({B}^\sig  )\times M^{2}(\ell ^{2})$
follows by Theorem \ref{thm:prp.lev}. We also observe that $Y\in \Sscr
^{2}$. Indeed, the $\Fbb ^{L}$-martingale $N$ satisfying the identity
$N_{t}=\Ebb \big [\xi +\int _{0}^{T}f(s,R_{s},S_{s},P_{s})\rmd s|\Fscr
^{L}_{t}]$ a.s., $t\in [0,T]$, is square integrable. So, using Doob's
inequality and Assumption \ref{ass:gen} (ii), from \eqref{eq:def.Y} we
also get the estimate
\begin{equation*}
\begin{split}
\Ebb \big [
\textstyle
\sup _{t\in [0,T]}Y_{t}^{2}\big ]
&\leq 2\Ebb \big [
\textstyle
\sup _{t\in [0,T]}N_{t}^{2}\big ]+
\displaystyle
8T\|\gamma \|^{2}_{L^{2}(\leb \otimes \Pbb )}
\\
&+16K^2T\Big (\|R\|_{L^{2}(\leb \otimes \Pbb )}+\|S\|^{2}_{L^{2}({B}^\sig  )}+
\|P\|^{2}_{M^{2}(\ell ^{2})}\Big )<+\infty
\end{split}
\end{equation*}
meaning that $Y\in \Sscr ^{2}$. In particular, we have $Y\in L^{2}(\leb
\otimes \Pbb )$. Finally, we observe that $(Y,Z,V)$ satisfies the
relation
\begin{equation*}
Y_{t}=\xi +\int _{t}^{T}f(s,R_{s},{S_{s},P_{s}})\rmd s-\int _{t}^{T}Z
_{s}\rmd {B}^\sig  _{s}-\sum _{n=1}^{\infty }V^{n}_{s}\rmd X^{f_{n}}
_{s}
\end{equation*}
and hence it satisfies \eqref{eq:bsde} if and only if it is a fixed
point of $\Phi $. We now define a $\beta $-norm $\|\cdot \|_{\beta }$
on $\Kscr ^{2}$ equivalent to $\|\cdot \|_{\Kscr ^{2}}$ with respect to
which $\Phi $ is a strong contraction: For any $(R,S,P)\in \Kscr ^{2}$
we define
\begin{equation*}
\|(R,S,P)\|_{\beta }:={\left (\Ebb \left [\int _{0}^{T}\rme ^{\beta s}
\left (R_{s}^{2}+\sigma ^{2}S^{2}_{s}+\|P_{s}\|_{\ell ^{2}}^{2}\right )\rmd s\right ]\right )
^{1/2}},\quad \beta >0.
\end{equation*}
We introduce the notation $(Y^{\prime },Z^{\prime },V^{\prime }):=
\Phi (R^{\prime },S^{\prime },P^{\prime })$ and
\begin{equation*}
(\overline{Y},\overline{Z},\overline{V}):=(Y-Y^{\prime },Z-Z^{\prime
},V-V^{\prime }),
\qquad
(\overline{R},\overline{S},\overline{P}):=(R-R^{\prime },S-S^{\prime
},P-P^{\prime }).
\end{equation*}
Applying twice integration by parts to $\rme ^{\beta t}\overline{Y}
_{t}^{2}$, because $\overline{Y}_{T}=0$, yields
%
\begin{equation}
\label{eq:it.for.app}
\rme ^{\beta t}\overline{Y}_{t}^{2}=-\beta \int _{t}^{T}\rme ^{\beta s}
\overline{Y}_{s-}^{2}\rmd s-2\int _{t}^{T}\rme ^{\beta s}\overline{Y}
_{s-}\rmd \overline{Y}_{s}-\int _{t}^{T}\rme ^{\beta s}\rmd [
\overline{Y},\overline{Y}]_{s}.
\end{equation}
We now compute $\rmd \overline{Y}_{s}$ and $\rmd [\overline{Y},
\overline{Y}]_{s}$. From \eqref{eq:def.Y} and \eqref{eq:def.Z.U}, for
$s\in (t,T]$, we deduce
%
\begin{equation}
\label{eq:com.dif}
\begin{split}
-\rmd \overline{Y}_{s}
&=\big (f(s,R_{s},S_{s},P_{s})-f(s,R_{s}^{\prime
},S_{s}^{\prime },P_{s}^{\prime })\big )
\rmd s-\overline{Z}_{s}\rmd
{B}^\sig  _{s}-\sum _{n=1}^{\infty }\overline{V}_{s}^{n}\rmd X^{f
_{n}}_{s}.
\end{split}
\end{equation}
Hence,
%
\begin{equation}
\label{eq:com.pb}
\rmd [\overline{Y},\overline{Y}]_{s}=\overline{Z}_{s}^{2}\rmd [{B}
^\sig  ,{B}^\sig  ]_{s}+\sum _{n,m=1}^{\infty }\overline{V}_{s}
^{n}\overline{V}_{s}^{m}\rmd [X^{f_{n}},X^{f_{m}}]_{s}.
\end{equation}
Inserting \eqref{eq:com.dif} and \eqref{eq:com.pb} in
\eqref{eq:it.for.app} gives
%
\begin{align}
\label{eq:it.for.app.pb}
\rme ^{\beta t}\overline{Y}_{t}^{2}=
&-\beta \int _{t}^{T}\rme ^{\beta
s}\overline{Y}_{s-}^{2}\rmd s+2\int _{t}^{T}\rme ^{\beta s}
\overline{Y}_{s-}\big (f(s,R_{s},S_{s},P_{s})-f(s,R_{s}^{\prime },S
_{s}^{\prime },P_{s}^{\prime })\big )\rmd s\nonumber
\\
&-2\int _{t}^{T}\rme ^{\beta s}\overline{Y}_{s-}\overline{Z}_{s}\rmd
{B}^\sig  _{s}-2\sum _{n=1}^{\infty }\int _{t}^{T}\rme ^{\beta s}
\overline{Y}_{s-}\overline{V}_{s}^{n}\rmd X^{f_{n}}_{s}\nonumber
\\
&
-\int _{t}^{T}\rme ^{\beta s}\overline{Z}_{s}^{2}\rmd [{B}^\sig  ,
{B}^\sig  ]_{s}-\sum _{n,m=1}^{\infty }\int _{t}^{T}\rme ^{\beta s}
\overline{V}_{s}^{n}\overline{V}_{s}^{m}\rmd [X^{f_{n}},X^{f_{m}}]_{s}.
\end{align}
Because of $Y,Y^{\prime }\in \Sscr ^{2}$ and $(R,S,P), (R^{\prime },S
^{\prime },P^{\prime })\in \Kscr ^{2}$, from Assumption \ref{ass:gen}
(ii) and the Cauchy--Schwarz inequality, the drift part in
\eqref{eq:it.for.app.pb} is integrable. Furthermore, by
the Kunita--Watanabe inequality, observing that $\aPP{\Xvtex{f_{n}}}{\Xvtex{f_{m}}}
_{t}= t\delta _{nm}$, $t\in [0,T]$, $\delta _{nm}$ denoting the Kronecker
symbol, since $V$ and ${V^{\prime }}$ belong to $M^{2}(\ell ^{2})$, we
get that the process $\sum _{n,m=1}^{\infty }\int _{0}^{\cdot }\rme
^{\beta s}\overline{V}_{s}^{n}\overline{V}_{s}^{m}\rmd [X^{f_{n}},X
^{f_{m}}]_{s}$ is of integrable variation. Let
$\Hscr ^{1}$ denote the space of local martingales $X$ such that $\|X\|_{\Hscr
^{1}}:=\Ebb [\sup _{t\in [0,T]}|X_{t}|]<+\infty $ and $X_{0}=0$. We
notice that $(\Hscr ^{1},\|\cdot \|_{\Hscr ^{1}})$ is a Banach space of
uniformly integrable martingales. The local martingales $\int _{0}^{
\cdot }\rme ^{\beta s}\,\overline{Y}_{s-}\,\overline{Z}_{s}\,\rmd
{B}^\sig  _{s}$ and $\int _{0}^{\cdot }\rme ^{\beta s}\,\overline{Y}
_{s-}\,\overline{V}^{n}_{s}\,\rmd X^{f_{n}}_{s}$ belong to $(\Hscr
^{1},\|\cdot \|_{\Hscr ^{1}})$. Indeed, since $\overline{Z}\in L^{2}(
{B}^\sig  )$, $\overline{V}\in M^{2}(\ell ^{2})$ and $\overline{Y}
\in \Sscr ^{2}$, by the Burkholder--Davis--Gundy inequality (see, e.g.,
\cite[Theorem 2.34]{J79}) and the Cauchy--Schwarz inequality, we see
that the estimates
\begin{align*}
\Ebb \left [\sup _{t\in [0,T]}\bigg |\int _{0}^{t}\rme ^{\beta s}\,
\overline{Y}_{s-}\,\overline{Z}_{s}\,\rmd {B}^\sig  _{s}\bigg |\right ]&<+
\infty \quad \textnormal{and}\\
\Ebb \left [\sup _{t\in [0,T]}
\bigg |\int _{0}^{t}\rme ^{\beta s}\,\overline{Y}_{s-}\,\overline{V}
^{n}_{s}\,\rmd X^{f_{n}}_{s}\bigg |\right ]&<+\infty
\end{align*}
hold. Analogously, since $\overline{V}\in M^{2}(\ell ^{2})$, it follows
that
\begin{equation*}
\begin{split}
\Ebb \bigg [\sup _{t\in [0,T]}\bigg |
&\sum _{n=k+1}^{k+m}\int _{0}^{t}\rme
^{\beta s}\overline{Y}_{s-}\overline{V}_{s}^{n}\rmd X^{f_{n}}_{s}
\bigg |\bigg ]
\\
&\leq c\rme ^{\beta T}\left (\Ebb \bigg [\sup _{t\in [0,T]}{\overline{Y}
^{2}_{t}}\bigg ]\right )^{1/2}\left (\Ebb \bigg [\sum _{n=k+1}^{k+m}\int
_{0}^{T}(\overline{V}_{s}^{n})^{2}\rmd s\bigg ]\right )^{1/2}\longrightarrow
0
\end{split}
\end{equation*}
as $k,m\rightarrow +\infty $, where $c>0$ is the constant coming from
the Burkholder--Davis--Gundy inequality. Therefore, we obtain that the sum
$\sum _{n=1}^{\infty }\int _{0}^{\cdot }\rme ^{\beta s}\overline{Y}_{s-}
\overline{V}_{s}^{n}\rmd X^{f_{n}}_{t}$ converges in $(\Hscr ^{1},\|
\cdot \|_{\Hscr ^{1}})$ and hence it is, in particular, a centered\vadjust{\goodbreak}
uniformly integrable martingale. Taking the expectation in
\eqref{eq:it.for.app.pb} now yields
%
\begin{align}
\label{eq:it.for.app.pb.ex}
\Ebb \bigg [\rme ^{\beta t}\overline{Y}_{t}^{2}\,{+}\,\sigma ^{2}\!&\int _{t}
^{T} \!\rme ^{\beta s}\overline{Z}_{s}^{2}\rmd s\,{+}\,\int _{t}^{T}\!\rme ^{
\beta s}\|\overline{V}\|^{2}_{\ell ^{2}}\rmd s\bigg ]\nonumber
=\\&-\beta \Ebb \bigg [\!\int _{t}^{T}\!\rme ^{\beta s}\overline{Y}_{s-}^{2}\rmd s
\bigg ]\nonumber
\\& \,{+}\,2\Ebb \bigg [\!\int _{t}^{T}\!\rme ^{\beta s}\overline{Y}_{s-}\big (f(s,R
_{s},S_{s},P_{s})\,{-}\,f(s,R_{s}^{\prime },S_{s}^{\prime },P_{s}^{\prime })
\big )\rmd s\bigg ].
\end{align}
We now use the abbreviation
\begin{equation*}
I:=2\Ebb \bigg [\int _{t}^{T}\rme ^{\beta s}\overline{Y}_{s-}\big (f(s,R
_{s},S_{s},P_{s})-f(s,R_{-}^{\prime },S_{s}^{\prime },P_{s}^{\prime })
\big )\rmd s\bigg ].
\end{equation*}
By Assumption \ref{ass:gen} (iii), we get
%
\begin{equation}
\label{eq:est.I}
|I|\leq 2C\Ebb \bigg [\int _{t}^{T}\rme ^{\beta s}\overline{Y}_{s-}
\big (|\overline{R}_{s}|+|\sigma |\,|\overline{S}_{s}|+\|\overline{P}
_{s}\|_{\ell ^{2}}\big )\rmd s\bigg ].
\end{equation}
Since the filtration $\Fbb ^{L}$ is quasi-left
continuous, the identity $\Ebb [Y_{s}^{2}]=\Ebb [Y_{s-}^{2}]$ holds. So, $\Ebb [\int _{t}^{T}\rme ^{\beta s}
\overline{Y}_{s-}^{2}\rmd s]= \Ebb [\int _{t}^{T}\rme ^{\beta s}
\overline{Y}_{s}^{2}\rmd s]$ holds by Fubini's theorem. For $h>0$ and $a,b,c\geq 0$, we have the estimates $ab\le a
^{2}\frac{h}{2}+\frac{b^{2}}{2h}$ and $(a+b+c)^{2}\leq 4(a^{2}+b^{2}+c
^{2})$. Applying these inequalities to \eqref{eq:est.I}, and then
choosing $h=8C$, we get
\begin{equation*}
|I|\leq 8C^{2}\,\Ebb \bigg [\int _{t}^{T}\rme ^{\beta s}\overline{Y}
_{s}^{2}\rmd s\bigg ]+\frac{1}{2}\,\Ebb \bigg [\int _{t}^{T}\rme ^{
\beta s}\big (|\overline{R}_{s}|^{2}+\sigma ^{2}\,|\overline{S}_{s}|^{2}+
\|\overline{P}_{s}\|_{\ell ^{2}}^{2}|^{2}\big )\rmd s\bigg ].
\end{equation*}
Using the latter estimate in \eqref{eq:it.for.app.pb.ex} and then taking
$\beta =8C^{2}+1$, we obtain
\begin{equation*}
\|(Y,Z,V)-(Y^{\prime },Z^{\prime },V^{\prime })\|_{8C^{2}+1}^{2}
\leq \frac{1}{2} \|(R,S,P)-(R^{\prime },S^{\prime },P^{\prime })\|
_{8C^{2}+1}^{2}
\end{equation*}
 which means that $\Phi $ is a strong contraction on $(\Kscr ^{2},\|
\cdot \|_{\beta })$ if $\beta =8C^{2}+1$. Hence, $\Phi $ has a unique
fixed point in $(\Kscr ^{2},\|\cdot \|_{8C^{2}+1})$. The proof of the
theorem is complete.
\end{proof}

\section{Logarithmic utility maximization\index{logarithmic utility maximization}}%
\label{sec:log.ut}

\subsection{The market model}%
\label{subs:ju.case}

Let $(L,\Fbb ^{L})$ be a L\'{e}vy process\index{L\'{e}vy process} with characteristics
$(\eta ,\sigma ^{2},\nu )$. We assume $\sigma ^{2}>0$ and consider the
probability space $(\Omega ,\Fscr ^{L}_{T},\Pbb )$.\index{probability space} We denote by
$\mu $ the jump measure of $L$ and set $\overline{\mu }:=\mu -\leb
\otimes \nu $.

Let $b$ and $\zeta $ be bounded predictable processes.\index{predictable ! process} We assume
furthermore that there exist $\varepsilon _{1}>\varepsilon _{2}>0$ such
that $\varepsilon _{2}\leq \zeta _{t}^{2}(\omega )\leq \varepsilon _{1}$,
for every $(t,\omega )\in [0,T]\times \Omega $. Additionally, let
$\beta $ be a bounded predictable function on $\widetilde{\Omega }$ such
that $\beta ({t,\omega },y)\ge 0$ and $\beta ({t,\omega },y)\leq
\alpha _{t}(\omega )(|y|\wedge 1)$, for all $({t,\omega },y)$ in
$\widetilde{\Omega }$, where $\alpha $ is a bounded nonnegative
$\Fbb ^{L}$-predictable process. By the assumptions on $\beta $, we have
\begin{equation*}
\Ebb \bigg [\sum _{0\leq s\leq T}\beta ^{2}(s,\Delta L_{s})1_{\{\Delta L
_{s}\neq 0\}}\bigg ]\leq \Ebb \bigg [\int _{0}^{T}\int _\Rbb \alpha _{s}
^{2}(y^{2}\wedge 1)\nu (\rmd y)\rmd s\bigg ]<+\infty \,.\vadjust{\goodbreak}
\end{equation*}
Therefore, the process $U=(U_{t})_{t\in [0,T]}$ defined by
%
\begin{equation}
\label{eq:def.U}
U_{t}:=\int _{0}^{t}b_{s}\rmd s+\int _{0}^{t}\zeta _{s}\rmd {B}^\sig
_{s}+\int _{[0,t]\times \Rbb }\beta (s,y)\overline{\mu }(\rmd s,\rmd y),
\quad t\in [0,T]\,,
\end{equation}
is a well-defined semimartingale and we can consider the price process\index{price process}
$S=S_{0}\,\Escr (U)$. Because of the assumptions on $\beta $ and from
the explicit expression of the stochastic exponential (see
\cite[Eq.\ I.4.64]{JS00}), it follows that $S>0$ and $S_{-}>0$.
Furthermore, by the Dol\'{e}ans-Dade equation, for every $t\in[0,T]$, the price process\index{price process}
$S$ satisfies
%
\begin{equation}
\label{eq:pr.pr.ju}
S_{t}=S_{0}+\int _{0}^{t}S_{s-}(b_{s}\rmd s+\zeta _{s}\rmd {B}^\sig
_{s})+\int _{[0,t]\times \Rbb }S_{s-}\beta (s,y)\overline{\mu }(\rmd s,\rmd y).
\end{equation}
Clearly, since $\beta (t,\omega ,y)\ge 0$ by assumption, the price
process $S$\index{price process} can only have positive jumps.

As in \cite{DS94}, the admissible strategies\index{admissible strategies} for the market model,
described by the locally bounded semimartingale $S$, are the predictable
processes\index{predictable ! process} $\pi $ such that the stochastic integral $\int _{0}^{\cdot }
\pi _{s}\rmd S_{s}$\index{stochastic integral} is a well-defined semimartingale and $\int _{0}^{T}
\pi _{s}\rmd S_{s}$ is bounded from below. An admissible strategy\index{admissible strategies}
$\pi $ is an arbitrage opportunity if it holds $\int _{0}^{T}\pi _{s}\rmd S
_{s}\geq 0$ a.s.\ and $\Pbb [{\int _{0}^{T}} \pi _{s}\rmd S_{s}>0]>0$. It
is well known that, in this context, the existence of an equivalent
local martingale measure for $S$ implies the absence of arbitrage
opportunities\index{arbitrage opportunities} (see \cite[Corollary 1.2]{DS94}).

We introduce the bounded predictable process $\theta :=\zeta ^{-1}b$. Let
now $\Qbb $ be the measure defined on $\Gscr _{T}$ by $ \rmd \Qbb :=\Escr (-
{\int _{0}^{\cdot }\theta _{s}/\sigma ^{2}\rmd {B}^\sig  }\big )_{T}\rmd \Pbb
$. By Novikov's condition, $\Qbb $ is a probability measure on
$\Gscr _{T}$ equivalent to $\Pbb $. According to Girsanov's theorem,
$
\widehat{B}^{\sigma }_{t}:={B}^\sig  _{t}+\int _{0}^{t}\theta _{s}\rmd s$,
$t\in [0,T]$, defines a $\Qbb $-Brownian motion with respect to
$\Gbb $. Under $\Qbb $, we consider the process $\widehat{U}$ defined
by
%
\begin{equation}
\label{def:proc.Rhat}
\widehat{U}_{t}:=\int _{0}^{t}\zeta _{s}\rmd \widehat{B}^{\sigma }_{s}+
\int _{[0,t]\times \Rbb }\beta (s,y)\overline{\mu }(\rmd s,\rmd y)\,.
\end{equation}
Therefore, under $\Qbb $, we get $S=S_{0}\,\Escr (\widehat{U})$. We are
now going to show that $S$ is a $\Qbb $-martingale and, hence, that the
market model is free of arbitrage opportunities.\index{arbitrage opportunities}

We denote by BMO$(\Qbb )$ the space of adapted BMO martingales with
respect to $\Qbb $ on $[0,T]$.

\begin{proposition}
\label{prop:ar.free}
Let $\Qbb $ be the equivalent probability measure defined above.

\textnormal{(i)} The $\Qbb $-compensator of the $\Pbb $-Poisson random
measure $\mu $ coincides with $\leb \otimes \nu $. Hence, $\mu $ is also
a $\Qbb $-Poisson random measure relative to $\Gbb $.

\textnormal{(ii)} The process $\widehat{U}$ from
\eqref{def:proc.Rhat} belongs to $\emph{BMO}(\Qbb )$.

\textnormal{(iii)} The price process\index{price process} $S$ is a martingale with respect
to $\Qbb $. In particular, the market model is free of arbitrage
opportunities.\index{arbitrage opportunities}
\end{proposition}
\begin{proof}
To verify (i), we observe that the density process of $\Qbb $ with
respect to $\Pbb $ is a continuous $\Gbb $-martingale. Hence, from
\cite[Theorem 12.31]{HWY92}, we conclude that the
$\Qbb $-compensator of $\mu $ coincides with $\leb \otimes \nu $.
Therefore, $\mu $ is a Poisson random measure relative to $\Gbb $ with
respect to $\Qbb $ (see \cite[Theorem II.4.8]{JS00}) and this
proves (i). We verify (ii). From (i) and the assumptions on
$\beta $, $\widehat{U}\in \Hscr ^{2}(\Qbb )$ holds. Furthermore, since
$\Delta \widehat{U}_{t}(\omega )=\beta (t,\omega ,\Delta L_{t}(\omega
))1_{\{\Delta L_{t}(\omega )\neq 0\}}$, $\beta $ being bounded,
$\widehat{U}$ has bounded jumps. Setting $c_{1}:=\sup _{t\in [0,T]}|
\zeta _{t}|$ and $c_{2}:=\sup _{t\in [0,T]}|\alpha _{t}|$, we have
\begin{equation*}
0<\aPP{\widehat U}{\widehat U}_{t}\leq \left (\sigma ^{2} c_{1}^{2}+c
_{2}^{2}\int _\Rbb (y^{2}\wedge 1)\nu (\rmd y)\right )T=:C(T)
,
\quad t\in [0,T]\,.
\end{equation*}
Hence, $\aPP{\widehat U}{\widehat U}$ is bounded on $[0,T]$. So,
\cite[Theorem 10.9 (2)]{HWY92} yields $\widehat{U}\in
\textnormal{BMO}(\Qbb )$. We now come to (iii). Under $\Qbb $, we have
$S=S_{0}\,\Escr (\widehat{U})$. Since $\Delta \widehat{U}\geq 0$, by
(ii) we can apply \cite[Theorem 2]{IS79}, which yields that $S$ is
a uniformly integrable $\Fbb ^{L}$-martingale under $\Qbb $. The proof
of the proposition is complete.
\end{proof}

\subsection{The optimization problem}

We now study the following optimization problem:
%
\begin{equation}
\label{eq:opt.log.ct}
V(x)=\sup _{\rho \in \Ascr }\Ebb \Big [\log \big (W_{T}^{\rho ,x}\big )
\Big ],\quad x>0\,,
\end{equation}
where $\Ascr $ and $W^{\rho ,x}>0$ (both to be defined) represent the
sets of admissible strategies\index{admissible strategies} and the wealth process\index{wealth process} with initial
capital $x>0$, respectively.

We are going to solve \eqref{eq:opt.log.ct} by constructing a family of
processes $\{R^{\rho ,x},\ \rho \in \Ascr \}$ fulfilling the
\emph{martingale optimality principle\index{martingale optimality principle}} on $[0,T]$ (see
\cite[p. 1697]{HIM05}).

\begin{assumption}[Martingale Optimality Principle]
\label{ass:mar.opt.pr}
Suppose that $x>0$. The family $\{R^{\rho ,x},\ \rho \in \Ascr \}$ is
$\Fbb ^{L}$-adapted and has the following properties:

(1) $R^{\rho ,x}_{T}=\log (W^{\rho ,x}_{T})$ for every
$\rho \in \Ascr $.

(2) $R^{\rho ,x}_{0}\equiv r^{x}$ is a constant not depending on
$\rho $ for every $\rho \in \Ascr $.

(3) $R^{\rho ,x}$ is a supermartingale for every $\rho \in \Ascr $.

(4) There exists $\rho ^{\ast }\in \Ascr $ such that $R^{\rho ^{\ast },x}$
is a martingale.
\end{assumption}
Notice that if the family $\{R^{\rho ,x},\ \rho \in \Ascr \}$ satisfies
Assumption \ref{ass:mar.opt.pr}, then the strategy $\rho ^{\ast }$ in
Assumption \ref{ass:mar.opt.pr} (4) is a solution of
\eqref{eq:opt.log.ct}. Indeed, for any $\rho \in \Ascr $, we get
\begin{equation*}
\Ebb \big [\log (W^{\rho ,x}_{T})]=\Ebb \big [R^{\rho ,x}_{T}\big ]
\leq R^{\rho ,x}_{0}=r^{x}=\Ebb \big [R^{\rho ^{\ast },x}_{T}\big ]=\Ebb
\big [\log (W^{\rho ^{\ast },x}_{T})\big ]\,.
\end{equation*}

We now define the set $\Ascr $ of admissible strategies\index{admissible strategies} and the wealth
process\index{wealth process} $W^{\rho ,x}$. For $\rho \in \Ascr $, we want to consider the
wealth process\index{wealth process} $W^{\rho ,x}$ given by
%
\begin{equation}
\label{eq:stoch.ex.wea.ju}
W^{\rho ,x}=x\,\Escr \left (\int _{0}^{\cdot }\rho _{s}\rmd U_{s}\right )=x+
\int _{0}^{\cdot }W^{\rho ,x}_{s-}\rho _{s}\rmd U_{s},\quad x>0\,,
\end{equation}
where the process $U$ is defined by \eqref{eq:def.U}. Therefore, we
assume that $\rho $ is a predictable process\index{predictable ! process} such that $\int _{0}^{T}
\rho ^{2}_{s}\rmd s<+\infty $ a.s.\ To ensure $W^{\rho ,x}>0$, we
assume $\rho _{t}(\omega )\geq 0$ which is, in particular, a short-sell
constraint on the admissible strategies.\index{admissible strategies}

\begin{definition}[Admissible strategies\index{admissible strategies}]
\label{def:adm.st.ju.co}
Let $C\neq \emptyset $ be a closed subset of $[0,+\infty )$. The set
$\Ascr $ of admissible strategies\index{admissible strategies} consists of all predictable and
$C$-valued processes $\rho $ satisfying the integrability condition
$\Ebb [\int _{0}^{T}\rho _{s}^{2}\rmd s]<+\infty $.
\end{definition}

In the following proposition, we summarize some properties of the
process $W^{\rho ,x}$ for $\rho \in \Ascr $.
%
\begin{proposition}
\label{prop:prop.weal.ju}
Let $\rho \in \Ascr $ (see Definition \ref{def:adm.st.ju.co}), $x>0$,
and $\beta (\omega ,t,y)\ge 0$, for every $(\omega ,t,y)\in
\widetilde{\Omega }$. The wealth process\index{wealth process} $W^{\rho ,x}$ is the
$\Gbb $-semimartingale given by the identity
\eqref{eq:stoch.ex.wea.ju}. Furthermore, for every $t\in [0,T]$, we have
$W^{\rho ,x}_{t}>0$ and
%
\begin{align}
\label{eq:smart.dec.wea.ju}
\log \big (W_{t}^{\rho ,x}\big )-\log x
&=\int _{0}^{t}\rho _{s}\zeta _{s}\rmd
{B}^\sig  _{s}
+\int _{[0,t]\times \Rbb }\log \big (1+\rho _{s}\beta (s,y)
\big )\overline{\mu }(\rmd s,\rmd y)\nonumber
\\
&+\int _{[0,t]\times \Rbb }
\big (\log (1+\rho _{s}\beta (s,y))-\rho _{s}
\beta (s,y)\big )\nu (\rmd y)\rmd s\nonumber
\\
&-\int _{0}^{t}\Big [\frac{\sigma ^{2}}{2}\Big (\rho _{s}\zeta _{s}-\frac{
\theta _{s}}{\sigma ^{2}}\Big )^{2}+\frac{\theta _{s}^{2}}{2\sigma ^{2}}
\Big ]\rmd s\,.
\end{align}
The local martingale part of $\log (W^{\rho ,x})$ is a true martingale
and $\log (W_{t}^{\rho ,x})$ is integrable, $t\in [0,T]$.
\end{proposition}
\begin{proof}
It is clear that for every $\rho \in \Ascr $, the wealth process\index{wealth process}
$W^{\rho ,x}$ is a semimartingale and it satisfies
\eqref{eq:stoch.ex.wea.ju}. Hence, since $\beta ({t,\omega },y)\ge 0$,
we obtain $W^{\rho ,x}_{t}(\omega )>0$, for every $\rho \in \Ascr $.
Thus, we can consider the process $\log (W^{\rho ,x})$. From
\eqref{eq:stoch.ex.wea.ju}, \eqref{eq:def.U} and the explicit expression
of the stochastic exponential, it follows
%
\begin{align}
\label{eq:rep.X.aux}
\log \big (W_{t}^{\rho ,x}&\big )-\log x
=\int _{0}^{t}\rho _{s}\zeta _{s}\rmd
{B}^\sig  _{s}
+\int _{[0,t]\times \Rbb }\rho _{s}\beta (s,y)\overline{
\mu }(\rmd s,\rmd y)\nonumber
\\
&
+\int _{0}^{t}\Big [-\frac{\sigma ^{2}}{2}\Big (\rho _{s}\zeta _{s}-\frac{
\theta _{s}}{\sigma ^{2}}\Big )^{2}+\frac{\theta _{s}^{2}}{2\sigma ^{2}}
\Big ]\rmd s\nonumber
\\
&
+\sum _{0\leq s\leq t}\big \{\log (1+\rho _{s}\beta (s,\Delta L_{s}))-
\rho _{s}\beta (s,\Delta L_{s})\big \}1_{\{\Delta L_{s}\neq 0\}}.
\end{align}
We define the processes
\begin{align*}
&A:=\sum _{0\leq s\leq \cdot }\big \{\log (1+\rho _{s}\beta (s,\Delta L
_{s}))-\rho _{s}\beta (s,\Delta L_{s})\big \}1_{\{\Delta L_{s}\neq 0\}},
\\
& B:=\sum _{0\leq s\leq \cdot }\big |\log (1+\rho _{s}\beta (s,\Delta L
_{s}))-\rho _{s}\beta (s,\Delta L_{s})\big |1_{\{\Delta L_{s}\neq 0\}}.
\end{align*}
The increasing process $B$ is integrable. To see this, we first recall
the following estimates:
%
\begin{equation}
\label{eq:int.cond.log.com}
|\log (1+y)|\leq |y|,
\qquad
|\log (1+y)-y|\leq y^{2},
\quad
\textnormal{for }\ y\geq 0\,.
\end{equation}
From \eqref{eq:int.cond.log.com}, since $\rho \in \Ascr $ and
$\rho _{t}(\omega )\beta (\omega ,t,y)\geq 0$, by the boundedness of
$\alpha $ and the assumptions on $\beta $, we get the estimate
%
\begin{equation}
\label{eq:int.drift}
\Ebb [B_{T}]\leq \Ebb \bigg [\int _{0}^{T}\rho _{s}^{2}\alpha _{s}^{2}\rmd s
\bigg ]\int _\Rbb (y^{2}\wedge 1)\nu (\rmd y)<+\infty \,.
\end{equation}
So, we can introduce $\int _{[0,\cdot ]\times \Rbb } \big (\log (1+\rho
_{s}\beta (s,y))-\rho _{s}\beta (s,y)\big )\mu (\rmd s,\rmd y)$ which is
a process of integrable variation and indistinguishable from $A$. Hence,
the predictable compensator\index{predictable ! compensator} of $A$ is $A^{p}:=
\int _{[0,\cdot ]\times \Rbb } \big (\log (1+\rho _{s}\beta (s,y))-\rho
_{s}\beta (s,y)\big )\nu (\rmd y)\rmd s$ and, according to
\cite[Proposition II.1.28]{JS00}, the identity
\begin{equation*}
A-A^{p}=\int _{[0,\cdot ]\times \Rbb }
\big (\log (1+\rho _{s}\beta (s,y))-
\rho _{s}\beta (s,y)\big )\overline{\mu }(\rmd s,\rmd y)
\end{equation*}
holds. In conclusion, by the linearity of the stochastic integral\index{stochastic integral} with
respect to $\overline{\mu }$, we can rewrite \eqref{eq:rep.X.aux} as in
\eqref{eq:smart.dec.wea.ju}. It remains to show that $\log (W^{\rho ,x}
_{t})$ is integrable for every $t\in [0,T]$. We observe that, because
of \eqref{eq:int.drift}, the boundedness of $\theta $ and $\zeta $,
since $\rho \in \Ascr $, the drift part in
\eqref{eq:smart.dec.wea.ju} is integrable. Furthermore, the local
martingale part of $\log (W^{\rho ,x})$ belongs to $\Hscr ^{2}$. Indeed,
by the boundedness of $\zeta $, we get that ${\int _{0}^{\cdot }\rho
_{s}\zeta _{s}\,\rmd {B}^\sig  _{s}}$ belongs to $\Hscr ^{2}$. From
\eqref{eq:int.cond.log.com} we have
\begin{equation*}
\log (1+\rho _{t}\beta (t,y))\leq \rho _{t}\beta (t,y)\leq \rho _{t}
\alpha _{t}(|y|\wedge 1)\in L^{2}(\leb \otimes \Pbb \otimes \nu )\,.
\end{equation*}
Thus, the stochastic integral\index{stochastic integral} with respect to $\overline{\mu }$ in
\eqref{eq:smart.dec.wea.ju} also belongs to $\Hscr ^{2}$. The proof
is complete.
\end{proof}
We notice that, for every $\rho \in \Ascr $, we have the identity
%
\begin{equation}
\label{eq:weal.gen}
W^{\rho ,x}_{t}=x+\int _{0}^{t}
\textstyle
{\frac{W^{\rho ,x}_{s-}\rho _{s}}{S_{s-}}}
\displaystyle
\,\rmd S_{s},\quad x>0\,,
\end{equation}
where $S$ is the price process\index{price process} of the stock. So, we can interpret an
admissible strategy\index{admissible strategies} $\rho \in \Ascr $ as the part of the wealth invested
in the stock, and  $\pi :={W^{\rho ,x}_{-}\rho }/{S_{-}}$
is the number of shares of the stock. Since from Proposition
\ref{prop:prop.weal.ju} the wealth process\index{wealth process} $W^{\rho ,x}$ is a positive
semimartingale, for every $\rho \in \Ascr $, the predictable process\index{predictable ! process}
$\pi $ is an admissible strategy\index{admissible strategies} for the market model described by the
price process $S$\index{price process} (see \eqref{eq:pr.pr.ju}).

We now state a measurable selection result, which will be useful in the
proof of Theorem \ref{thm:log.ut.max.con.ju} below.

\begin{lemma}
\label{lem:fil.im.fct}
Let $C\subseteq \Rbb $ be a closed subset and let $G:{[0,T]\times
\Omega }\times C\longrightarrow \Rbb $ be a mapping such that:

\textnormal{(i)} $c\mapsto G({t,\omega },c)$ is continuous over $C$ for
every $({t,\omega })\in {[0,T]\times \Omega }$.

\textnormal{(ii)} $({t,\omega })\mapsto G({t,\omega },c)$ is an
$\Gbb $-predictable process for every $c\in C$.

\textnormal{(iii)} For all $({t,\omega })\in {[0,T]\times \Omega }$,
there exists $c^{\ast }\in C$ such that
\begin{equation*}
G({t,\omega },c^{\ast })=\inf _{c\in C} G({t,\omega },c).
\end{equation*}

Then, the infimum is, in fact, the minimum, $({t,\omega })\mapsto
\min _{c\in C} G({t,\omega },c)$ is a predictable process\index{predictable ! process} and there
exists a predictable process\index{predictable ! process} $\rho ^{\ast }$ such that
\begin{equation*}
G({t,\omega },\rho ^{\ast }_{t}(\omega ))=\min _{c\in C} G({t,\omega },c),
\quad
\textit{ for every } ({t,\omega })\in {[0,T]\times \Omega }.
\end{equation*}
\end{lemma}
\begin{proof}
Clearly, by (iii), the infimum is the minimum and $({t,\omega })
\mapsto \min _{c\in C} G({t,\omega },c)$ is a predictable process.\index{predictable ! process}
Indeed, denoting by $\mathbf{Q}$ the \emph{set of rational numbers},
from the continuity of $G$, $\inf _{c\in C} G({t,\omega },c)=
\inf _{c\in C\cap \mathbf{Q}} G({t,\omega },c)$,\vadjust{\goodbreak} for $({t,\omega })
\in {[0,T]\times \Omega }$, and the second claim is proven. The third
claim follows from assumption (iii) and Filippov's implicit function
theorem as formulated in \cite[Theorem 21.3.4]{CE15} (with $U=C$ and
without the space $X$). The proof of the lemma is complete.
\end{proof}

We are now ready to solve \eqref{eq:opt.log.ct}.
%
\begin{theorem}
\label{thm:log.ut.max.con.ju}
Let $S$ be the price process\index{price process} given in \eqref{eq:pr.pr.ju} with the
additional assumption $\beta ({t,\omega },y)\ge 0$, $({t,\omega },y)
\in \widetilde{\Omega }$. Let $\Ascr $ be as in Definition
\ref{def:adm.st.ju.co}. Then, for every $\rho \in \Ascr $, the wealth
process\index{wealth process} $W^{\rho ,x}$ satisfies \eqref{eq:stoch.ex.wea.ju} and
$\log \big (W^{\rho ,x}_{T}\big )\in L^{1}(\Fscr _{T}^{L},\Pbb )$.
Furthermore, for $ x>0$, the explicit expression of the value function
$V$ of the optimization problem \eqref{eq:opt.log.ct} is $V(x)=\log (x)+\Ebb [
\int _{0}^{T}f(s)\rmd s]$, where $f$ given by
%
\begin{align}
\label{eq:ass.gen.ju}
f({t,\omega })&:=-\min _{c\in C}
\bigg (\frac{\sigma ^{2}}{2}\Big (c\zeta
_{t}(\omega )-\frac{\theta _{t}(\omega )}{\sigma ^{2}}\Big )^{2}\nonumber
\\
&\quad\quad+\int _\Rbb \big \{c\beta ({t,\omega },y)-\log (1+c\beta ({t,\omega
},y))\big \}\nu (\rmd y)\bigg )+\frac{\theta ^{2}_{t}(\omega )}{2\sigma
^{2}}\,.
\end{align}
Moreover, there exists an admissible strategy\index{admissible strategies} $\rho ^{\ast }\in \Ascr $
such that, for every $({t,\omega })$ in $[0,T]\times \Omega $,
%
\begin{align}
\label{eq:op.str.ju}
\rho _{t}^{\ast }(\omega )\in \argmin_{c\in C}\bigg (
&
\frac{\sigma ^{2}}{2}\Big (c\,\zeta _{t}(\omega )-\frac{\theta _{t}(
\omega )}{\sigma ^{2}}\Big )^{2}\nonumber
\\
&+\int _\Rbb \big \{c\beta ({t,\omega },x)-\log (1+c\beta (t,{\omega
},x))\big \}\nu (\rmd x)\bigg )
\end{align}
holds and $\rho ^{\ast }$ is optimal, that is, $V(x)=\Ebb [\log (W_{T}
^{\rho ^{\ast },x})]$, $x>0$.
\end{theorem}
\begin{proof}
We only need to verify the statements about the optimization problem
\eqref{eq:opt.log.ct}, since the properties of the wealth process\index{wealth process}
$W^{\rho ,x}$ come from Proposition \ref{prop:prop.weal.ju}. We prove
the result in two steps. First, using Lemma \ref{lem:fil.im.fct}, we
show that $({t,\omega })\mapsto f({t,\omega })$ in
\eqref{eq:ass.gen.ju} is an admissible generator and that there exists
a $\rho ^{\ast }\in \Ascr $ satisfying \eqref{eq:op.str.ju}. We then show
the optimality of the strategy $\rho ^{\ast }\in \Ascr $ as an
application of the martingale optimality principle.\index{martingale optimality principle}

For each $c$ in the closed subset $C\subseteq [0,+\infty )$, we define
\begin{equation*}
\begin{split}
G({t,\omega },c):=\frac{\sigma ^{2}}{2}\big (c\zeta _{t}(\omega )-
&
\sigma ^{-2}\theta _{t}(\omega )\big )^{2}
\\
&+\int _\Rbb \big \{c\beta ({t,\omega },y)-\log (1+c\beta ({t,\omega
},y))\big \}\nu (\rmd y)\,.
\end{split}
\end{equation*}
Since $\beta $ is a predictable function,\index{predictable function} the process $({t,\omega })
\mapsto G({t,\omega },c)$ is predictable, for every $c\in C$. We now
show the continuity of ${c}\mapsto G({t,\omega },c)$ on $C$ for
$({t,\omega })$ in ${[0,T]\times \Omega }$. Let $(c^{n})_{n\in \Nbb }
\subseteq C$ be a convergent sequence and let $c\in C$ be its limit. We
have
\begin{align*}
&c^{n}\beta ({t,\omega },y)-\log (1+c^{n}\beta ({t,\omega },y))\\
&\quad \longrightarrow
c\beta ({t,\omega },y)-\log (1+c\beta ({t,\omega },y)),\quad n\rightarrow
+\infty \,,
\end{align*}
pointwise in $t$, $\omega $ and $y$. From \eqref{eq:int.cond.log.com},
we get, as $n\rightarrow +\infty $,
\begin{equation*}
0\leq c^{n}\beta ({t,\omega },y)-\log (1+c^{n}\beta ({t,\omega },y))
\leq (c^{n}\alpha _{t}(\omega ))^{2}(y^{2}\wedge 1)\longrightarrow c
\alpha _{t}^{2}(\omega )(y^{2}\wedge 1)\,.
\end{equation*}
By dominated convergence, we get $G({t,\omega },c^{n})\longrightarrow
G({t,\omega },c)$, $n\rightarrow +\infty $, which is the statement about
the continuity. We now show that there exists a $c^{\ast }$ in $ C$ such
that $G({t,\omega },c^{\ast })=\min _{c\in C}G({t,\omega },c)$, for every
$({t,\omega })$ in ${[0,T]\times \Omega }$. By the estimate
$G({t,\omega },c)\geq \frac{\sigma ^{2}}{2}\big (c\zeta _{t}(\omega )-
\sigma ^{-2}\theta _{t}(\omega )\big )^{2}$, we get $G({t,\omega },c)
\longrightarrow +\infty $ as $c\rightarrow \infty $, for every
$({t,\omega })$ in ${[0,T]\times \Omega }$. Hence,
$C_{0}:=\{c\in C: G({t,\omega },c)\leq G({t,\omega },c_{0})\}$ is a closed bounded set and, consequently, compact, where $c_{0}\in C$ is chosen
arbitrarily but fixed. By the continuity of $G({t,\omega },\cdot )$,
there exists $c^{\ast }{=c^{\ast }(t,\omega )}$ in $C_{0}$ such that
$G({t,\omega },c^{\ast })=\min _{c\in C_{0}}G({t,\omega },c)$ holds for
every $(t,\omega )$ in $[0,T]\times \Omega $. We also have $G({t,
\omega },c^{\ast })=\inf _{c\in C}G({t,\omega },c)$ for every
$({t,\omega })\in {[0,T]\times \Omega }$. So, by Lemma
\ref{lem:fil.im.fct}, we get that $({t,\omega })\mapsto \min _{c\in C}G(
{t,\omega },c)$ is a predictable process\index{predictable process} and that there exists a
$C$-valued predictable $\rho ^{\ast }$ such that, for every
$(t,\omega )$ in $[0,T]\times \Omega $, the identity $G({t,\omega },\rho ^{\ast }
_{t}(\omega ))=\min _{c\in C}G({t,\omega },c)$ holds.

We now show that
$\Ebb [\int _{0}^{T}(\rho _{s}^{\ast })^{2}\rmd s]<+\infty $. By $0<\beta ({t,\omega ,y})\leq \alpha
_{t}(\omega )(|y|\wedge 1)$ and
\eqref{eq:int.cond.log.com}, for $c\in C$, we can estimate
\begin{equation*}
\label{eq:est.ju.G}
\begin{split}
\int _\Rbb \big (c\beta (t,y)-\log \big (1+c\beta (t,y)\big )\big )
\nu (\rmd y)\leq \left (\int _\Rbb (y^{2}\wedge 1)\nu (\rmd y)\right )
\alpha ^{2}_{t}c^{2}<+\infty \,.
\end{split}
\end{equation*}
So, $\alpha $ being bounded, we get $0\leq G({t,\omega },c)\leq k_{1}c
^{2}+k_{2}$, where $k_{1},k_{2}>0$ denote two suitable constants. Using
the boundedness of $\zeta $ and $\theta $, the minimality property of
$\rho ^{\ast }$ and the estimate \eqref{eq:est.ju.G}, it is therefore
straightforward to see that, for two suitable constants $\overline{k}
_{1},\overline{k}_{2}>0$, we have $|\rho _{t}^{\ast }(\omega )|\leq
\overline{k}_{1}G({t,\omega },c)+\overline{k}_{2}$ for every
$c\in C$. Hence, we get $ (\rho _{t}^{\ast }(\omega ))^{2}\leq
\tilde{k}_{1} c^{2}+\tilde{k}_{2} $, $c\in C$, where $\tilde{k}_{1},
\tilde{k}_{2}>0$ are suitable constants. This implies $\rho ^{\ast }
\in \Ascr $. We now verify that $f$ in \eqref{eq:ass.gen.ju} satisfies
Assumption \ref{ass:gen} (i) and (ii). Because of the previous step and
the predictability of $\theta $, from the identity
\begin{equation*}
f({t,\omega })=-\min _{c\in C}G({t,\omega },c)+\frac{\theta ^{2}_{t}(
\omega )}{2\sigma ^{2}},\quad t\in [0,T],
\end{equation*}
we deduce that $({t,\omega })\mapsto f({t,\omega })$ is predictable.
This shows that $f$ fulfils Assumption \ref{ass:gen} (i). The estimate
\begin{equation*}
|f({t,\omega })|\leq G({t,\omega },c)+\frac{\theta _{t}^{2}(\omega )}{2
\sigma ^{2}}\leq k_{1}c^{2}+k,\quad c\in C,
\end{equation*}
where $k>0$ is a suitable constant, implies that $f$ satisfies
Assumption \ref{ass:gen} (ii).

We now construct a family of processes $\{R^{\rho ,x},\ \rho \in \Ascr
\}$ which satisfies Assumption \ref{ass:mar.opt.pr}.

{Notice that, because $f$ satisfies Assumptions \ref{ass:mar.opt.pr} (i)
and (ii), the process $\int _{0}^{\cdot }f(s)\rmd s$ is
$\Fbb ^{L}$-adapted, $f$ being predictable and Lebesgue integrable.
Hence, we can consider the square integrable martingale $N$ satisfying
$N_{t}=\Ebb [\int _{0}^{T}f(s)\rmd s|\Fscr _{t}^{L}]$ a.s., $t\in [0,T]$.
We define the \cadlag \ semimartingale $Y=(Y_{t})_{t\in [0,T]}$ by
setting
%
\begin{equation}
\label{eq:proc.Y}
Y_{t}:=N_{t}-\int _{0}^{t} f(s)\rmd s,\quad t\in [0,T].
\end{equation}
We observe that, $\Fscr _{0}^{L}$ being trivial, we have $Y_{0}=N_{0}=\Ebb [
\int _{0}^{T}f(s)\rmd s]$. Furthermore, $Y$ satisfies $Y_{t}=\Ebb [\int
_{t}^{T}f(s)\rmd s|\Fscr _{t}^{L}]$ a.s., $t\in [0,T]$ and $Y_{T}=0$.

We now set $R^{\rho ,x}_{t}:=\log (W^{\rho ,x}_{t})+Y_{t}$ for
$t\in [0,T]$. Notice that $R^{\rho ,x}$ fulfils Assumption
\ref{ass:mar.opt.pr} (1), since $R^{\rho ,x}_{T}:=\log (W^{\rho ,x}
_{T})$ holds. From Proposition \ref{prop:prop.weal.ju}, for every
$t$ in $[0,T]$, we get
%
\begin{align}
\label{eq:expr.Rrho}
R^{\rho ,x}_{t}
&=\log (x)+N_{t}+\int _{0}^{t}\rho _{s}\zeta _{s}\rmd
{B}^\sig  _{s}
{+}\int _{[0,t]\times \Rbb }\log \big (1+\rho _{s}
\beta (s,y)\big )\overline{\mu }(\rmd s,\rmd y)\nonumber
\\
&-\int _{0}^{t}\Big \{f(s)+\frac{\sigma ^{2}}{2}\Big (\rho _{s}\zeta _{s}-\frac{
\theta _{s}}{\sigma ^{2}}\Big )^{2}\nonumber
\\
&
\hspace{1cm}+\int _{\Rbb }
\big (\rho _{s}\beta (s,y)-\log (1+\rho _{s}
\beta (s,y))\big )\nu (\rmd y)-\frac{\theta _{s}^{2}}{2\sigma ^{2}}
\Big \}\rmd s\,{.}
\end{align}
The first line on the right-hand side of \eqref{eq:expr.Rrho} consists
of true martingales. Because the drift part on the right-hand side of
\eqref{eq:expr.Rrho} is non-positive and integrable, $R^{\rho ,x}$ is
a supermartingale for every $\rho \in \Ascr $. Additionally,
$R^{\rho ,x}_{0}$ does not depend on $\rho $. Furthermore, if
$\rho ^{\ast }$ is the admissible strategy\index{admissible strategies} introduced above, then
$\rho ^{\ast }$ satisfies \eqref{eq:op.str.ju} and $R^{\rho ^{\ast },x}$
is a true martingale. The martingale optimality principle\index{martingale optimality principle} implies the
optimality of $\rho ^{\ast }$. Hence, $V(x)=\Ebb [\log (W_{T}^{\rho
^{\ast },x})]=\log (x)+Y_{0}$ and the proof of the theorem is
complete.}
\end{proof}
%

%
\begin{remark}
\label{rem:BSDEs}
It is evident from the first part of the proof of Theorem
\ref{thm:log.ut.max.con.ju}, that the predictable function $f$ defined
in \eqref{eq:ass.gen.ju} is an admissible generator. So, because of
Theorem \ref{thm:ex.un}, the BSDE
%
\begin{equation}
\label{eq:ass.bsde.ju}
\widetilde{Y}_{t}=0+\int _{t}^{T} f(s)\rmd s-\int _{t}^{T}Z_{s}\rmd
{B}^\sig  _{s}-\sum _{n=1}^{\infty }\int _{t}^{T} V^{n}_{s}\rmd X^{f
_{n}}_{s}
\end{equation}
has a unique solution $(\widetilde{Y},Z,V)\in \Sscr ^{2}\times \Lrm
^{2}({B}^\sig  )\times M^{2}(\ell ^{2})$, where $(Z,V)$ is the unique
pair such that for every $t\in [0,T]$
\begin{equation*}
\begin{split}
N_{t}
&=\Ebb \left [\int _{0}^{T}f(s)\rmd s\Big |\Fscr _{t}^{L}\right ]
\\
&=\Ebb \left [\int _{0}^{T}f(s)\rmd s\right ]+\int _{0}^{t}Z_{s}\rmd
{B}^\sig  _{s}+\sum _{n=1}^{\infty }\int _{0}^{t}V_{s}^{n}\rmd X^{f
_{n}}_{s},
\end{split}
\end{equation*}
holds and $\widetilde{Y}_{t}=\Ebb [\int _{t}^{T}f(s)\rmd s|\Fscr _{t}
^{L}]$. Clearly, $\widetilde{Y}$ satisfies $
\widetilde{Y}_{t}=N_{t}-\int _{0}^{t}f(s)\rmd s$, for every $t$ in $[0,T]$, and
$\widetilde{Y}_{0}=\Ebb [\int _{0}^{T}f(s)\rmd s]$. Hence, $
\widetilde{Y}=Y$, where $Y$ has been defined in \eqref{eq:proc.Y}. This
shows that the martingale optimality principle\index{martingale optimality principle} in Theorem
\ref{thm:log.ut.max.con.ju} can be also constructed as an application
of Theorem \ref{thm:ex.un}.
\end{remark}
%



\begin{funding}
PDT gratefully acknowledges Martin Keller-Ressel and funding from the \gsponsor[id=GS1,sponsor-id=501100001659]{German Research Foundation} (DFG) under grant \gnumber[refid=GS1]{ZUK 64}.
\end{funding}

\end{document}